\documentclass{article}
\usepackage[margin=2.5cm]{geometry}
\usepackage{amsmath, amsthm, amssymb}
\usepackage{txfonts}
\usepackage{authblk}
\usepackage{amsfonts}
\usepackage{hyperref}

\newtheorem{theorem}{Theorem}[section]
\newtheorem{corollary}{Corollary}[section]
\newtheorem{lemma}{Lemma}[section]
\newtheorem{definition}{Definition}[section]
\theoremstyle{definition} \newtheorem{remark}{Remark}[section]
\theoremstyle{definition} \newtheorem{example}{Example}[section]
\theoremstyle{definition} \newtheorem{question}{Question}[section]

\begin{document}
\title{The Homology Groups of Zero Divisor Graphs of Finite Commutative Rings}
\author[1]{Fenglin Li\\ \href{mailto: fenglin125@126.com}{fenglin125@126.com}}
\affil[1]{Department of Information Technology, Zhejiang Financial College}
\date{\today}
\maketitle
\abstract{This paper investigates the homology groups of the clique complex associated with the zero-divisor graph of a finite commutative ring. Generalizing the construction introduced by F. R. DeMeyer and L. DeMeyer, we establish a K\"unneth-type formula for the homology of such complexes and provide explicit computations for products of finite local rings. As a notable application, we obtain a general method to determine the clique homology groups of $\mathbb{Z}_n$ and related ring products. Furthermore, we derive explicit formulas for the Betti numbers when all local factors are fields or non-fields. A complete classification of when this clique complex is Cohen–Macaulay is given, with the exception of one borderline case. Finally, our results yield a partial answer to a question posed in earlier literature, showing that certain topological spaces, such as the Klein bottle and the real projective plane, cannot be realized as zero-divisor complexes of finite commutative rings.}
\footnote{supported by the Fundamental Research Funds for the Provincial Universities of Zhejiang(Grant Number: 2025PY01)}

\section{Introduction}
In this paper, a graph is an undirected simple graph and a ring is a unitary commutative ring. The study of zero-divisor graphs was initiated by Beck \cite{beck1988coloring}, who defined the graph $G(R)$ for a commutative ring $R$ with $1 \neq 0$ to have all elements of $R$ as vertices, where two distinct vertices $a$ and $b$ are adjacent if $ab = 0$. Anderson and Livingston \cite{AL99} later introduced a refinement, the graph $\Gamma(R)$, defined as the induced subgraph of $G(R)$ on the set of nonzero zero-divisors $Z(R)^* = Z(R) \setminus \{0\}$. The interplay between the ring structure of $R$ and the graph structure of $\Gamma(R)$ has since been extensively studied \cite{AL99, AB08, red06}.

This concept extends naturally to commutative semigroups $S$ \cite{DMS02}, and was further developed by DeMeyer and DeMeyer \cite{demeyer05}, who introduced the simplicial complex $K(S)$—the clique complex of $\Gamma(S)$. Recall that the clique complex $\Delta(G)$ of a graph $G$ is the simplicial complex whose simplices are the cliques of $G$. By a slight abuse of notation, we write $K(R)$ for the clique complex of $\Gamma(R)$. The primary objective of this paper is to conduct a systematic investigation of the simplicial homology groups of $K(R)$ for finite commutative rings $R$. We also utilize the complex $K_0(S)$, where a simplex is a subset $A \subseteq S \setminus \{0\}$ such that $xy = 0$ for all distinct $x, y \in A$.

Computing the homology of such zero-divisor complexes presents a fundamental challenge whose tractability depends crucially on the underlying algebraic structure. For certain highly symmetric cases, such as chessboard complexes (arising from products of fields) or matching complexes, the homology can be attacked via group actions and combinatorial representation theory. However, for a general commutative semigroup, zero-divisors lack a uniform description, and no such global symmetry is available; consequently, a universal homological formula appears out of reach.

The case of finite commutative rings occupies a privileged intermediate position. The Chinese Remainder Theorem provides a canonical decomposition into local rings, endowing the set of zero-divisors with a rich yet rigid product structure. It is precisely this rigidity that makes a complete homological analysis possible. For a finite commutative ring $\Omega$, the classical structure theorem yields an isomorphism $\Omega \cong R_1 \times R_2 \times \cdots \times R_k$, where each $R_i$ is a finite local ring. Our first main contribution is a K\"unneth-type formula that expresses the homology group $H_n(K(\Omega))$ in terms of data derived from the local factors $R_i$. The structure of $K(\Omega)$ is intimately related to the join operation of simplicial complexes. In the simplest case, $K(F_1 \times F_2)$ is the complete bipartite graph $K_{u_1, u_2}$, whose clique complex is well-understood. To handle the general case, we introduce the following key construction:

\begin{definition} \label{main-def}
Let $A$ be a subcomplex of a simplicial complex $K$ and $L$ be another simplicial complex. The join of $L$ and $K$ over $A$ is a simplicial complex denoted by $K\triangledown_{A}L$, whose vertex set is the disjoint union of the vertex sets of $K$ and $L$. The simplices of $K\triangledown_{A}L$ are those of $K$, those of $L$, and the disjoint unions of simplices of $A$ with simplices of $L$.
\end{definition}

This definition provides an efficient recursive description of $K(\Omega)$, which serves as the foundation for our homological computations.

In Section 2, we review necessary background in simplicial homology and establish our primary technical tool (Theorem \ref{tool2}), which describes the homology of complexes of the form $X \triangledown_A \overline{K}_r$. Its proof leverages the K\"unneth theorem and the Mayer–Vietoris sequence.

Section 3 is devoted to our central result, Theorem \ref{main}, which gives an explicit, recursive formula for the homology groups $\tilde{H}_n(K(\Omega))$. This theorem allows us to compute these groups inductively. As a first application, we obtain a partial answer to a question posed in \cite{demeyer05}, showing that neither the Klein bottle nor the real projective plane can be realized as $K(\Omega)$ for any finite commutative ring $\Omega$. Furthermore, we derive explicit formulas for the Betti numbers in the cases where all local factors are fields, and where none of them are fields.

In Section 4, we apply our results to two classification problems. First, we provide a near-complete characterization of the rings $\Omega$ for which $K(\Omega)$ is a Cohen–Macaulay complex, with only one borderline case remaining open. Second, we prove that $K(\Omega)$ cannot be a triangulation of any compact surface, thus strengthening the aforementioned topological obstructions.

Finally, in Section 5, we discuss natural open questions, including the extension of our results to finite noncommutative rings and the complete classification of the Cohen–Macaulay property in the remaining case.

\section{Preliminaries}
We first review some definitions and results of simplicial homology theory that will be used in this paper. \par

An abstract simplicial complex $K$ is a vertex set $\{ v_{\alpha}\}$ and a family of its subsets, called simplices, such that any subset of a simplex must be a simplex. A set of $k+1$ vertices is a $k-$dimensional simplex. A subcomplex $L$ of $K$ is an abstract simplicial complex such that every simplex of $L$ belong to $K$. The formal sum $\sum a_i \Delta_i^k$ is called an $n-$chain, where $a_i$ is an integer and $\Delta_i^k$ is a simplex of dimension $k$. The group of $k-$chains is denoted by $C_k(K)$. Assume that all of simplices of $K$ are oriented, i.e., the vertices of each simplex are ordered. A simplex with vertices with order $a_0,a_1,\ldots,a_n$ is denoted by $[a_0,a_1,\ldots,a_n]$. We define the boundary of a simplex as
\[ \partial[0,1,\ldots,n]=\sum(-1)^i[0,\ldots,\hat{i},\ldots,n], \]
where $[0,\ldots,\hat{i},\ldots,n]=[0,\ldots,i-1,i+1,\ldots,n]$. Extending $\partial$ by linearity, we obtain a manp $\partial_k\colon C_k(K)\rightarrow C_{k-1}(K)$. Since $\partial\partial=0$, we have in fact defined a complex of abelian groups
\[ \cdots\rightarrow C_k(K)\rightarrow C_{k-1}(K)\rightarrow\cdots\rightarrow C_0(K)\rightarrow 0. \]
The quotient group $H_k(K)=\mathrm{ker}\partial_k / \mathrm{im}\partial_{k+1}$ is called the $k-$dimensional simplicial homology group of the complex $K$. \par

If we replace the map $\partial_0\colon C_0(K)\rightarrow0$ by $\epsilon\colon C_0(K)\rightarrow\mathbb{Z}$, where $\epsilon(v)=1$ for each vertex $v$, then we obtain a new chain complex. The homology group of this chain complex is called the reduced homology group of $K$, and is denoted by $\tilde{H}_k(K)$. In this paper, we mainly use the reduced homology group.

\begin{theorem}[K\"unneth Relation \cite{CE99}]\label{kun}
Let $L_*$ and $M_*$ be complexes of abelian groups. Then we have an exact sequence
\[ 0\rightarrow \sum_{i+j=n}H_{i}(L_*)\otimes H_{j}(M_*)\rightarrow H_{n}(L_{*}\otimes M_{*})\rightarrow\sum_{i+j=n-1}\mathrm{Tor}_{1}(H_{i}(L_*),H_{j}(M_*))\rightarrow 0. \]
\end{theorem}

\begin{theorem}\label{tool1}
Let $G_1$ and $G_2$ be simple graphs. Suppose that $\tilde{H}_*(\Delta(G_1))$ or $\tilde{H}_*(\Delta(G_2))$ is free, then we have
\[ \tilde{H}_n(\Delta(G_1 \triangledown G_2))\cong \sum_{i+j=n-1}\tilde{H}_i(\Delta(G_1))\otimes \tilde{H}_j(\Delta(G_2)).\]
\end{theorem}
\begin{proof}
Consider the augmentation complex $\tilde{C}_* (\Delta(G_1 \triangledown G_2))$. In the join $G_1 \triangledown G_2$, each vertex of $G_1$ is adjacent to each vertex of $G_2$. Then an $n-$simplex of $\Delta(G_1 \triangledown G_2)$ can be identified with an $i-$simplex of $\Delta(G_1)$ and a $j-$simplex of $\Delta(G_2)$, where $i+j=n-1$, here we view a $(-1)-$simplex as an empty set. Given a left complex $L_*$, we write $L_*[k]$ for the shift of $L_*$ with degree $k$, where $L_n[k]=L_{n-k}$ and $d_{n}[1]=(-1)^k d_{n-k}$. By the above argument, we have
\[ \tilde{C}_*(\Delta(G_1 \triangledown G_2))\cong(\tilde{C}_*(\Delta(G_1))[1]\otimes \tilde{C}_*(\Delta(G_2))[1])[-1]. \]
Applying the K\"unneth theorem (Theorem \ref{kun}) to the tensor product of the shifted complexes, we obtain:
\begin{align*}
\tilde{H}_n(\Delta(G_1 \triangledown G_2)) &\cong H_{n+1}(\tilde{C}_*(\Delta(G_1))[1]\otimes \tilde{C}_*(\Delta(G_2))[1]) \\
 &\cong\sum_{i+j=n+1}H_i(\tilde{C}_*(\Delta(G_1))[1])\otimes H_j(\tilde{C}_*(\Delta(G_2))[1]) \\
 &=\sum_{i+j=n+1}H_{i-1}(\tilde{C}_*(\Delta(G_1))\otimes H_{j-1}(\tilde{C}_*(\Delta(G_2)) \\
 &=\sum_{i+j=n-1}\tilde{H}_i(\Delta(G_1))\otimes \tilde{H}_j(\Delta(G_2)).
\end{align*}
The first isomorphism follows from the identification of the chain complex, the second from Theorem \ref{kun}, and the last equality is a reindexing.
\end{proof}

\begin{corollary}
Let $A$ be a simplical complex. Then
\[ \tilde{H}_n(A\triangledown\overline{K}_r)\cong[\tilde{H}_{n-1}(A)]^{r-1}. \]
\end{corollary}
\begin{proof}
Note that $\tilde{H}_n(\overline{K}_r)=0$ for $n\neq0$ and $\tilde{H}_0(\overline{K}_r)\cong\mathbb{Z}^{r-1}$.
\end{proof}

\begin{corollary}
Let $G$ be the complete $r-$partite graph $K_{m_1,m_2,\ldots,m_r}$. Then
\[ \tilde{H}_{r-1}(\Delta(G))\cong\mathbb{Z}^{(m_1 -1)(m_2 -1)\cdots(m_r -1)}, \]
and $\tilde{H}_n(\Delta(G))=0$ if $n\neq r-1$.
\end{corollary}
\begin{proof}
Note that $K_{m_1,m_2,\ldots,m_r}=\overline{K}_{m_1}\triangledown\overline{K}_{m_2}\triangledown\cdots\triangledown\overline{K}_{m_r}$, and the conclusion follows by induction.
\end{proof}

\begin{corollary}\label{cor3}
Let $F_1$ and $F_2$ be finite fields. Set $u_1=|F_1^*|$ and $u_2=|F_2^*|$. Then
\[ \tilde{H}_1(K(F_1 \times F_2))\cong\mathbb{Z}^{(u_1 -1)(u_2 -1)}, \]
and $\tilde{H}_n(K(F_1 \times F_2))=0$ if $n\neq 1$.
\end{corollary}
\begin{proof}
Observe that $K(F_1 \times F_2)=K_{u_1,u_2}$.
\end{proof}

\begin{theorem}[Mayer Vietoris \cite{pra07}]\label{mv}
Suppose that $K$ is a simplicial complex, $K_0$ and $K_1$ are subcomplexes of $K$ such that $K=K_0 \cup K_1$, and $L=K_0 \cap K_1$. Then there is an exact sequence
\[ \cdots\rightarrow H_k(L)\rightarrow H_k(K_0)\oplus H_k(K_1)\rightarrow H_k(K)\rightarrow H_{k-1}(L)\rightarrow\cdots . \]
\end{theorem}

\begin{theorem}\label{tool2}
Let $A$ be a subcomplex of a simplicial complex $X$. Suppose the natural homomorphism $\tilde{H}_*(A)\rightarrow\tilde{H}_*(X)$ is zero and $\tilde{H}_*(A)$ is free, then we have
\[ \tilde{H}_n(X\triangledown_{A}\overline{K}_r)\cong\tilde{H}_n(X)\oplus[\tilde{H}_{n-1}(A)]^r .\]
\end{theorem}
\begin{proof}
By Theorem \ref{mv}, we have an exact sequence
\[ \tilde{H}_n(A)\rightarrow\tilde{H}_n(X)\oplus\tilde{H}_n(A\triangledown\overline{K}_r)\rightarrow\tilde{H}_n(X\triangledown_{A}\overline{K}_r)\rightarrow\tilde{H}_{n-1}(A)\rightarrow\tilde{H}_{n-1}(X)\oplus\tilde{H}_{n-1}(A\triangledown\overline{K}_r). \]
Using Theorem \ref{kun} and Theorem \ref{tool1}, we can see that the map $\tilde{H}_n(A)\rightarrow\tilde{H}_n(A\triangledown\overline{K}_r)$ factors through $\tilde{H}_n(A)\rightarrow\tilde{H}_n(A)\otimes\tilde{H}_{-1}(\overline{K}_r)\rightarrow\tilde{H}_n(A\triangledown\overline{K}_r)$. Therefore the natural homomorphism $\tilde{H}_n(A)\rightarrow\tilde{H}_n(A\triangledown\overline{K}_r)$ is zero. Since we assume that $\tilde{H}_*(A)\rightarrow\tilde{H}_*(X)$ is zero, the above exact sequence reduces to a short exact sequence
\[ 0\rightarrow\tilde{H}_n(X)\oplus\tilde{H}_n(A\triangledown\overline{K}_r)\rightarrow\tilde{H}_n(X\triangledown_{A}\overline{K}_r)\rightarrow\tilde{H}_{n-1}(A)\rightarrow0. \]
By assumption the $\tilde{H}_*(A)$ is free, this short exact sequnce is split and conclusion follows.
\end{proof}

\section{Main Theorems}

In this section, $\Omega$ will denote a finte commutative ring, and $\Omega\cong R_1 \times R_2 \times\cdots\times R_k$, where $R_i$ is a finite local ring. We use $\mathfrak{m}_i$ to denote the maximal ideal of $R_i$. Let $v_i$ be the least positive integer such that $\mathfrak{m}_i^{v_i}= 0$. We use $U(R)$ to denote the set of units of a ring $R$ and we write $|U(R_i)|=u_i$.

\begin{lemma}
Let $R$ be a finite local ring. Then $\tilde{H}_n(K(R))=0$.
\end{lemma}
\begin{proof}
This is proved in \cite{demeyer05}, Corollary 6.
\end{proof}

\begin{lemma}\label{lem}
For a finite commutative ring $R$, we have $\tilde{H}_n(K_0(R))\cong\tilde{H}_n(K(R))$ for $n\neq0$, and
\[ \tilde{H}_0(K_0(R))=\begin{cases}
\mathbb{Z}^{|U(R)|}, & \text{if } R \text{ is not a field;} \\
\mathbb{Z}^{|U(R)|-1}, & \text{if } R \text{ is a field.}
\end{cases} \]
\end{lemma}
\begin{proof}
Recall that $K_0(R)$ is the disjoint union of $K(R)$ and the set of units $U(R)$ (as isolated vertices).

We compute the reduced homology dimension by dimension.

For $n > 0$, since the reduced homology of a discrete set of points vanishes in positive dimensions, and using the general fact that $\tilde{H}_n(X \sqcup Y) \cong \tilde{H}_n(X) \oplus \tilde{H}_n(Y)$ for $n > 0$, we have:
\[
\tilde{H}_n(K_0(R)) \cong \tilde{H}_n(K(R)) \oplus \tilde{H}_n(\overline{K}_{|U(R)|}) \cong \tilde{H}_n(K(R)) \oplus 0 \cong \tilde{H}_n(K(R)).
\]

For $n = 0$, we compute directly using the number of connected components. Let $c(X)$ denote the number of connected components of $X$.

Case 1: $R$ is not a field.
Then $K(R)$ is connected and non-empty, so $c(K(R)) = 1$. The set of units contributes $|U(R)|$ isolated vertices, giving:
\[
c(K_0(R)) = c(K(R)) + |U(R)| = 1 + |U(R)|.
\]
Therefore, $\tilde{H}_0(K_0(R)) \cong \mathbb{Z}^{(1 + |U(R)|) - 1} = \mathbb{Z}^{|U(R)|}$.

Case 2: $R$ is a field.
Then $K(R)$ is empty and $K_0(R)$ is just the discrete set of units, so $c(K_0(R)) = |U(R)|$ and:
\[
\tilde{H}_0(K_0(R)) \cong \mathbb{Z}^{|U(R)|-1}.
\]

This completes the proof of the lemma.
\end{proof}

\begin{theorem} \label{main}
The homology group $\tilde{H}_n(K(\Omega))$ is free abelian for any integer $n\in\mathbb{Z}$, and we have
\[ \tilde{H}_n(K(\Omega))\cong\sum_{j=1}^{k-1}\sum_{1\leq i_1 <\cdots<i_j\leq k-1}\tilde{H}_{n-1}(K_0(R_{i_1}\times R_{i_2}\times\cdots\times R_{i_j}))^{u_1\cdots\hat{u}_{i_1}\cdots\hat{u}_{i_j}\cdots u_k}
\]
if $R_k$ is not a field. When $R_k$ is a field, the exponent of $\tilde{H}_{n-1}(K_0(R_1\times\cdots\times R_{k-1}))$ in the above formula should be changed to $u_k -1$.
\end{theorem}
\begin{proof}
We use induction on $k$. Represent the elements of $\Omega$ as $k$-tuples $(x_1,\ldots,x_k)$ with $x_i\in R_i$. Let $K_0$ be the subcomplex of $K(\Omega)$ induced on $\{(x_1,\ldots,x_k)\in Z(\Omega)^* \mid x_k\notin U(R_k)\}$. This complex is a cone: for any $z_k \in \mathfrak{m}_k^{v_k-1}\setminus \{0\}$, the vertex $(0,\ldots,0,z_k)$ is adjacent to all other vertices of $K_0$. Hence $\tilde{H}_n(K_0)=0$ for all $n\in\mathbb{Z}$.

Define $K_1$ as the subcomplex induced on $K_0$ together with $0\times\cdots\times0\times U(R_k)$. Each vertex in $0\times\cdots\times0\times U(R_k)$ is adjacent to the subcomplex on $(R_1\times\cdots\times R_{k-1}\setminus\{0\})\times 0$, which is isomorphic to $K_0(R_1\times\cdots\times R_{k-1})$. Since the clique complex of $0\times\cdots\times0\times U(R_k)$ is $\overline{K}_{u_k}$, we have $K_1 = K_0 \triangledown_{K_0(R_1\times\cdots\times R_{k-1})} \overline{K}_{u_k}$. Theorem \ref{tool2} gives:
\[ \tilde{H}_n(K_1) \cong \tilde{H}_{n-1}(K_0(R_1\times\cdots\times R_{k-1}))^{u_k}. \]

The remaining vertices form the set $Z(R_1\times\cdots\times R_{k-1})^* \times U(R_k)$. We partition them into two types:

\textbf{Type I:} Vertices with at least one nonzero zero-divisor in the first $k-1$ coordinates. These lie in sets:
\[
V_J = \left( \prod_{j \in J} Z(R_j)^* \times \prod_{i \notin J} R_i \right) \times U(R_k)
\]
for non-empty $J \subseteq \{1,\ldots,k-1\}$. For each such $V_J$, the attachment subcomplex $A_J$ is a cone (take $a_j \in \mathfrak{m}_j^{v_j-1}\setminus\{0\}$ for some $j \in J$; then $(0,\ldots,a_j,\ldots,0,0)$ is a cone point). By Theorem \ref{tool2}, these attachments preserve homology.

\textbf{Type II:} Vertices where the first $k-1$ coordinates are units or zero, but not all units. These correspond to sets:
\[
V_J' = \left( \prod_{j \in J} \{0\} \times \prod_{i \notin J} U(R_i) \right) \times U(R_k)
\]
for non-empty $J \subseteq \{1,\ldots,k-1\}$. The attachment is over $A_J'$ induced on 
\[\left( \prod_{j \in J} R_j \times \prod_{i \notin J} \{0\} \right) \times \{0\},\]
which is isomorphic to $K_0\left( \prod_{i \in J} R_i \right)$. Since $\tilde{H}_*(K_0) = 0$, the map $\tilde{H}_*(A_J') \to \tilde{H}_*(K_1)$ is zero. By Theorem \ref{tool2}, attaching $V_J'$ contributes a direct summand:
\[
\tilde{H}_{*-1}\left(K_0\left( \textstyle\prod_{i \in J} R_i \right)\right)^{u_k \cdot \prod_{i \in J} u_i}.
\]

After attaching all Type I and Type II vertices, we obtain $K(\Omega)$. Type I attachments preserve homology, while Type II attachments yield the additional terms in the formula. This completes the inductive step when $R_k$ is not a field.

Now assume $R_k$ is a field. Then:
\[
K_0 = K_0((R_1\times\cdots\times R_{k-1}\setminus\{0\})\times\{0\}) = K_0(R_1\times\cdots\times R_{k-1})
\]
and $K_1 = K_0(R_1\times\cdots\times R_{k-1}) \triangledown \overline{K}_{u_k}$. Corollary 2.1 gives:
\[
\tilde{H}_n(K_1) \cong \tilde{H}_{n-1}(K_0(R_1\times\cdots\times R_{k-1}))^{u_k-1}.
\]
Since $\tilde{H}_n(K_0) \to \tilde{H}_n(K_1)$ is zero, the same procedure applies. The conclusion follows.
\end{proof}

\begin{corollary} \label{cor4}
The homology group $\tilde{H}_n(K(\Omega))$ is 0 if $n\geq k$.
\end{corollary}

\begin{remark}
In \cite{demeyer05}, the authors ask if there is a simplicial decomposition of the Klein bottle or the real projective plane which is the complex of a semigroup $S$. From Theorem \ref{main} we know the answer is no if $S$ is a finite commutative ring.
\end{remark}

Using Theorem \ref{main}, we can determine the homology groups $\tilde{H}_n(K(\mathbb{Z}_n))$. We provide some examples.

\begin{example}
By Theorem \ref{main}, we have $\tilde{H}_n(K(\mathbb{Z}_{p_1^{n_1}}\times\mathbb{Z}_{p_2^{n_2}}))=\tilde{H}_{n-1}(K_0(\mathbb{Z}_{p_1^{n_1}}))^{\phi(p_2^{n_2})}$, where $n_1,n_2>1$. Therefore
\[ \tilde{H}_n(K(\mathbb{Z}_{p_1^{n_1}}\times\mathbb{Z}_{p_2^{n_2}}))=\begin{cases}
\mathbb{Z}^{\phi(p_1^{n_1})\phi(p_2^{n_2})}, & n=1; \\
0, & n\neq 1.
\end{cases} \]
\end{example}

\begin{example}
Let $n_1>1$. Since $\tilde{H}_n(K(\mathbb{Z}_{p_1^{n_1}}\times\mathbb{Z}_{p_2}))=\tilde{H}_{n-1}(K_0(\mathbb{Z}_{p_1^{n_1}}))^{p_2-2}$, we have
\[ \tilde{H}_n(K(\mathbb{Z}_{p_1^{n_1}}\times\mathbb{Z}_{p_2}))=\begin{cases}
\mathbb{Z}^{\phi(p_1^{n_1})(p_2-2)}, & n=1; \\
0, & n\neq 1.
\end{cases} \]
\end{example}

\begin{example}\label{ex3}
The Theorem \ref{main} implies that
 \begin{align*}  & \tilde{H}_n(K(\mathbb{Z}_{p_1}\times\mathbb{Z}_{p_2}\times\mathbb{Z}_{p_3})) \\
={} & \tilde{H}_{n-1}(K_0(\mathbb{Z}_{p_1}\times\mathbb{Z}_{p_2}))^{p_3 -2}\oplus\tilde{H}_{n-1}(K_0(\mathbb{Z}_{p_1}))^{(p_2 -1)(p_3 -1)}\oplus\tilde{H}_{n-1}(K_0(\mathbb{Z}_{p_2}))^{(p_1 -1)(p_2 -1)}. \end{align*}
Then $\tilde{H}_2(K(\mathbb{Z}_{p_1}\times\mathbb{Z}_{p_2}\times\mathbb{Z}_{p_3}))=\mathbb{Z}^{(p_1 -2)(p_2 -2)(p_3 -2)}$ and 
\[ \tilde{H}_1(K(\mathbb{Z}_{p_1}\times\mathbb{Z}_{p_2}\times\mathbb{Z}_{p_3}))=\mathbb{Z}^{(p_1-1)(p_2-1)(p_3-2)+(p_1-2)(p_2-1)(p_3-1)+(p_1-1)(p_2-2)(p_3-1)}. \]
For any other integer $n$, $\tilde{H}_n(K(\mathbb{Z}_{p_1}\times\mathbb{Z}_{p_2}\times\mathbb{Z}_{p_3}))=0$.
\end{example}

\begin{example}
Let $R\cong\mathbb{Z}_{p_1}\times\mathbb{Z}_{p_2}\times\mathbb{Z}_{p_3}\times\mathbb{Z}_{p_4}$. We compute $\tilde{H}_n(K(R))$. Since $\tilde{H}_n(K(\mathbb{Z}_{p_1}\times\mathbb{Z}_{p_2}\times\mathbb{Z}_{p_3}))=0$ for $n\geq 3$ and $\tilde{H}_n(K(\mathbb{Z}_{p_i}\times\mathbb{Z}_{p_j}))=0$ for $n\geq 2$. Hence $\tilde{H}_n(K(R))=0$ for $n\geq 4$ and $\tilde{H}_3(K(R))=\mathbb{Z}^{(p_1-2)(p_2-2)(p_3-2)(p_4-2)}$. By Example \ref{ex3} and Corollary \ref{cor3}, we get
\[ \mathrm{rank}\tilde{H}_2(K(R))=\sum_{1\leq i_1<i_2\leq 4}(p_{i_1}-1)(p_{i_2}-1)(p_{j_1}-2)(p_{j_2}-2), \]
where $\{j_1,j_2\}=\{1,2,3,4\}\setminus\{i_1,i_2\}$. By Lemma \ref{lem}, we have
\[ \mathrm{rank}\tilde{H}_1(K(R))=\sum_{1\leq i_1<i_2<i_3\leq 4}(p_{i_1}-1)(p_{i_2}-1)(p_{i_3}-1)(p_{j_1}-2)+3\prod_{k=1}^4 (p_k-1), \]
where $\{j_1\}=\{1,2,3,4\}\setminus\{i_1,i_2,i_3\}$.
\end{example}

When all the $R_i$ are fields or not fields, the Betti numbers $b_n(K(\Omega))$ can be viewed as a symmetric polynomial with variables $u_1,\ldots, u_k$. In these cases, we can describe the Betti numbers $b_n(K(\Omega))$ more explicitly. In the following, we use $b_n(u_1,\ldots, u_k)$ to denote the Betti number $\mathrm{rank}\tilde{H}_n(K(\Omega))$. By Corollary \ref{cor4}, we can focus on the case $n\leq k-1$.\par

When none of the local factors are fields, the Betti numbers take a particularly simple form.

\begin{theorem}
When every $R_i$ is not a field, we have
\[ b_n(u_1,\ldots, u_k)=a_{n,k}u_1\cdots u_k, \]
where $a_{n,k}\in\mathbb{Z}, a_{0,k}=1$ and
\[ a_{n,k}=\sum_{j=1}^{k-1}\binom{k-1}{j}a_{n-1,j}. \]
\end{theorem}
\begin{proof}
The statement $a_{0,k}=1$ follows from the Lemma \ref{lem}. We use induction to prove the second equation. By Theorem \ref{main}, we have
\begin{align*}
b_n(u_1,\ldots, u_k) &=\sum_{j=1}^{k-1}\sum_{1\leq i_1 <\cdots<i_j\leq k-1}b_{n-1}(u_{i_1},\dots,u_{i_j})\frac{u_1\cdots u_k}{u_{i_1}\cdots u_{i_j}} \\
&=\sum_{j=1}^{k-1}\sum_{1\leq i_1 <\cdots<i_j\leq k-1}a_{n-1,j}u_1\cdots u_k \\
&=\sum_{j=1}^{k-1}a_{n-1,j}\binom{k-1}{j} u_1\cdots u_k. 
\end{align*}
The induction completes.
\end{proof}

If all the $R_i$ are fields, the Theorem \ref{main} implies that the $b_n(u_1,\ldots, u_k)$ involves $u_i$ and $u_i -1$. We set
\[ \sigma_j(u_1,\ldots, u_k)\coloneqq \sum_{1\leq i_1 <\cdots<i_j\leq k}(u_{i_1}-1)\cdots(u_{i_j}-1)u_{n_1}\cdots u_{n_{k-j}}, \]
where $\{ n_1,\ldots, n_{k-j}\}=\{1,2,\ldots,n\}\setminus \{i_1,\ldots,i_j \}$. If $j>k$ or $j<0$, we understand that $\sigma_j(u_1,\ldots, u_k)=0$.

\begin{theorem}
If every $R_i$ is a field, then we have $b_{k-1}(u_1,\ldots,u_k)=\sigma_k(u_1,\ldots,u_k)$ and
\[ b_{k-j}(u_1,\ldots,u_k)=A_{j-1,1}\sigma_{k-j}(u_1,\ldots,u_k)+\cdots +A_{j-1,j-1}\sigma_{k-2j+2}(u_1,\ldots,u_k) \]
for any $j\geq 2$. The numbers $A_{i,j}$ satisfies the recurrence relation
\[ A_{j-1,t}=A_{j-2,t-1}\binom{j+t-2}{1}+A_{j-3,t-1}\binom{j+t-2}{2}+\cdots+A_{t-1,t-1}\binom{j+t-2}{j-t} \]
and $A_{j,1}=1$.
\end{theorem}
\begin{proof}
The equation $b_{k-1}(u_1,\ldots,u_k)=\sigma_k(u_1,\ldots,u_k)$ can be easily deduced by the Theorem \ref{main} and Corollary \ref{cor4}. Now we use double induction to prove the second equation. The Lemma \ref{lem} implies that $b_0(u_1,\ldots,u_j)=u_1\cdots u_j=\sigma_0(u_1,\ldots,u_j)$. This proves the case $k=j$. Next we establish the equation for $j=2$.  Using Theorem \ref{main} and induction on $k$, we have
\begin{align*}
b_{k-2}(u_1,\ldots,u_k) &=b_{k-3}(u_1,\ldots,u_{k-1})(u_k-1)+\sum_{j=1}^{k-1}b_{k-3}(u_1,\ldots,\hat{u}_j,\ldots,u_{k-1})u_j u_k \\
&=\sigma_{k-3}(u_1,\ldots,u_{k-1})(u_k -1)+\sum_{j=1}^{k-1}\sigma_{k-2}(u_1,\ldots,\hat{u}_j,\ldots,u_{k-1})u_j u_k \\
&=\sigma_{k-2}(u_1,\ldots,u_k).
\end{align*}
Set $V(i_1,\ldots,i_m)=\{u_1,\ldots,u_{k-1}\}\setminus \{u_{i_1},\ldots,u_{i_m}\}$. Therefore
\[ b_{k-j}(u_1,\ldots,u_k) =\sum_{m=0}^{j-1}\sum_{1\leq i_1 <\cdots<i_m\leq k-1}b_{k-j-1}(V(i_1,\ldots,i_m))u_{i_1}\cdots u_{i_m}u_{k,m}, \]
where $u_{k,0}=u_k -1$ and $u_{k,m}=u_k$ for $m\geq 1$. By induction on $j$ and $k$, we have the term $b_{k-j-1}(V(i_1,\ldots,i_m))$ equals to
\[ A_{j-m-1,1}\sigma_{k-j-1}(V(i_1,\ldots,i_m))+\cdots+A_{j-m-1,j-m-1}\sigma_{k-2j+m+1}(V(i_1,\ldots,i_m)). \]
If $m=0$, then the term $b_{k-j-1}(u_1,\ldots,u_{k-1})(u_k -1)$ contributes 
\[ A_{j-1,t}\sigma_{k-j-t}(u_1,\ldots,u_{k-1})(u_k -1). \]
If $m\geq 1$, the term $A_{j-m-1,t-1}\sigma_{k-j-t+1}(V(i_1,\ldots,i_m))$ contributes  \[ (v_1-1)\cdots(v_{k-j-t+1}-1)(\prod_{u_i \notin\{v_1,\ldots,v_{k-j-t+1},u_k\}}u_i) u_k \]
if and only if $\{v_1,\ldots,v_{k-j-t+1}\}\cap\{u_{i_1},\ldots,u_{i_m}\}=\emptyset$. There are $\binom{j+t-2}{m}$ ways to choose $u_1,\ldots,u_m$. Since $A_{j-1,t}=\sum_{m=1}^{j-t}A_{j-m-1,t-1}\binom{j+t-2}{m}$, the coefficient of the term
\[ (u_1-1)\cdots(u_{k-j-t+1}-1)\prod_{u_i \notin\{u_1,\ldots,u_{k-j-t+1}\}}u_i \]
is always $A_{j-1,t}$, no matter whether there exists the factor $(u_k -1)$. The induction completes.
\end{proof}

\section{Applications}

In this section, we give two applications of our main Theorem \ref{main}. We can classify the ring $\Omega$ when the simplicial complex $K(\Omega)$ is Cohen-Macaulay or $K(\Omega)$ is a decomposition of a compact surface.

\subsection{The Cohen-Macaulay Property of $K(\Omega)$}

Let $\Delta$ be a simplicial complex with vertex set $\{v_1,\ldots,v_n\}$ and $k$ a ring. Recall that the Stanley-Reisner ring of the complex $\Delta$ is the homogeneneous $k-$algebra $k[\Delta]=k[X_1,\ldots,X_n]/I_{\Delta}$, where $I_{\Delta}$ is the ideal generated by all monomials $X_{i_1}X_{i_2}\cdots X_{i_s}$ such that $\{v_{i_1},v_{i_2},\ldots,v_{i_s}\}\notin\Delta$. The complex $\Delta$ is called a Cohen-Macaulay complex over $k$ if $k[\Delta]$ is a Cohen-Macaulay ring. We say that $\Delta$ is a Cohen-Macaulay complex if $\Delta$ is Cohen-Macaulay over some field. Let $F$ be a subset of the vertex set of $\Delta$. The link of $F$ is the set $\mathrm{lk}_{\Delta}F=\{G\in\Delta\colon F\cup G\in\Delta,F\cap G=\emptyset\}$. For more information about Cohen-Macaulay complex, the reader can consult \cite{BH98}.

\begin{theorem}[\cite{rei76}]\label{reisner}
Let $\Delta$ be a simplicial complex, $k$ a field. The following conditions are equivalent: \\
(a) $\Delta$ is Cohen-Macaulay over $k$; \\
(b) $\tilde{H}_i(\mathrm{lk}F;k)=0$ for all $F\in\Delta$ and all $i<\dim\ \mathrm{lk}F$;\\
(c) $\tilde{H}_i(\Delta;k)=0$ for all $i<\dim\ \Delta$, and the links of all vertices of $\Delta$ are Cohen-Macaulay over $k$.
\end{theorem}

\begin{remark}\label{1d-cm}
Using condition (c) of Theorem \ref{reisner}, one can show that a simplicial complex of dimension 1 is Cohen-Macaulay (over any field) if and only if it is connected.
Indeed, if $\Delta$ is $1$-dimensional and connected, then $\tilde{H}_0(\Delta;k)=0$ and $\dim\Delta=1$, so the first condition in (c) holds. 
Moreover, the link of any vertex is $0$-dimensional (a discrete set of points), and any $0$-dimensional complex is Cohen-Macaulay. 
Thus $\Delta$ satisfies (c). The converse follows immediately from (c) since $\tilde{H}_0(\Delta;k)$ must vanish.
\end{remark}

\begin{theorem}\label{cm}
Let $\Omega$ be a finite commutative ring. The complex $K(\Omega)$ is Cohen--Macaulay if and only if one of the following holds:
\begin{enumerate}
    \item $\Omega$ is not local, and $\Omega \cong F_1 \times F_2$ or $\Omega \cong \mathbb{Z}_2[X]/(X^2) \times F_2$, where $F_1, F_2$ are finite fields;
    \item $\Omega$ is a local ring and is a field;
    \item $\Omega$ is a local ring (not a field) with maximal ideal $\mathfrak{m}$ of nilpotency index $v$, and $|\mathfrak{m}^{v-1}| > 2$;
    \item $\Omega$ is a local ring (not a field) with $|\mathfrak{m}^{v-1}| = 2$, and $K(\Omega)$ is Cohen--Macaulay (a full classification in this case remains open).
\end{enumerate}
\end{theorem}
\begin{proof}
We proceed by considering the value of $k$ in the decomposition $\Omega\cong R_1 \times R_2 \times\cdots\times R_k$.\\
Case 1: If $k\geq 3$, then $\dim K(\Omega)\geq k-1>1$ and $\tilde{H}_1(K(\Omega))\neq0$. By condition (c) of Theorem \ref{reisner}, the complex $K(\Omega)$ is not Cohen-Macaulay. \\
Case 2: Let $R_i$ denote a local ring that is not a field, and let $F_i$ denote a field. If $k=2$, we can classify the type of $\Omega$ into the following three cases: $F_1 \times F_2$, $R_1 \times R_2$, and $R_1 \times F_2$. If $\Omega\cong F_1\times F_2$, then $K(\Omega)$ is a connected 1-dimensional complex. By Remark \ref{1d-cm}, $K(\Omega)$ is Cohen-Macaulay. Now we discuss the case $\Omega\cong R_1\times R_2$. Let $\mathfrak{m}_i$ be the maximal ideal of $R_i$. Let $x\in\mathfrak{m}_1^{v_1-1}\setminus\{0\}$ and $y\in\mathfrak{m}_2^{v_2-1}\setminus\{0\}$. Then the complex $K(\Omega)$ has a 2-dimensional simplex $\{(x,0),(0,y),(x,y)\}$ and $\dim K(\Omega)\geq 2$. Note that $\tilde{H}_1(K(\Omega))=\mathbb{Z}^{u_1u_2}\neq0$, hence the complex $K(\Omega)$ is not Cohen-Macaulay. Finally we suppose $\Omega\cong R_1\times F_2$. The link of the vertex $(0,1)$ is $K_0(R_1)$. If $|\mathfrak{m}_1|\geq 3$, then $\dim K_0(R_1)>0$. Since $K_0(R_1)$ is not connected, it is not Cohen-Macaulay. If $|\mathfrak{m}_1|=2$, we first show that $R_1$ must be isomorphic to $\mathbb{Z}_2[X]/(X^2)$. Let $\mathfrak{m}_1=\{0,t\}$ and $R_1/\mathfrak{m}_1=\mathbb{F}_q$. The equation $t^2=t$ implies that $t(1-t)=0$. Since $1-t$ is a unit, we must have $t=0$, which is a contradiction. Hence $t^2=0$ and $\mathfrak{m}_1^2=0$. With this condition, the ideal $\mathfrak{m}_1$ can be viewed as a vector space over $R_1/\mathfrak{m}_1=\mathbb{F}_q$. This shows that $q\leq2$, and therefore $q=2$, which implies that $R_1\cong\mathbb{Z}_2[X]/(X^2)$. Note that the complex $K(\mathbb{Z}_2[X]/(X^2)\times F_2)$ is 1-dimensional and connected, and it is Cohen-Macaulay. \\
Case 3: We will show that the link of every vertex in $K(R_1)$ is a cone, which implies it is CM. Let $z \in Z(R_1)^*$ be an arbitrary vertex. Consider the elements $x, y \in \mathfrak{m}_1^{v_1-1} \setminus {0}$ with $x \neq y$. We claim that at least one of $x$ and $y$ is in the link of $z$, i.e., adjacent to $z$. Since $x, y, z \in \mathfrak{m}_1$, we have $xz, yz \in \mathfrak{m}_1^{v_1} = 0$. Thus, both $x$ and $y$ are adjacent to $z$ in the graph $\Gamma(R_1)$. However, for one of them to be in the link of $z$ in the complex $K(R_1)$, it must be a vertex distinct from $z$. If $z \neq x$ and $z \neq y$, then both $x$ and $y$ are in $\mathrm{lk}(z)$. If $z = x$, then $y$ is distinct from $z$ and is in $\mathrm{lk}(z)$. Similarly, if $z = y$, then $x$ is in $\mathrm{lk}(z)$. In all cases, there exists a vertex $p \in {x, y}$ such that $p \neq z$ and $p \in \mathrm{lk}(z)$. Now, since $p \in \mathfrak{m}_1^{v_1-1}$, we have $p w = 0$ for all $w \in \mathfrak{m}_1$, meaning $p$ is adjacent to every vertex in $K(R_1)$. In particular, $p$ is adjacent to every vertex in $\mathrm{lk}(z)$. Therefore, $\mathrm{lk}(z)$ is a cone with apex $p$, and hence is Cohen-Macaulay. By Theorem \ref{reisner}(c), since all vertex links are CM and $\tilde{H}_i(K(R_1)) = 0$ for $i < \dim K(R_1)$ (as $K(R_1)$ is itself a cone with apex $x$), the complex $K(R_1)$ is Cohen-Macaulay.
\end{proof}

\begin{remark}
In case 4, the complex $K(\Omega)$ can be either Cohen--Macaulay or not. For example:
\begin{itemize}
    \item $K(\Omega)$ is Cohen--Macaulay for $\Omega = \mathbb{Z}_2[X]/(X^3)$ and $\mathbb{Z}_2[X]/(X^4)$;
    \item $K(\Omega)$ is not Cohen--Macaulay for $\Omega = \mathbb{Z}_2[X]/(X^n)$ with $n \ge 5$;
    \item $K(\Omega)$ is Cohen--Macaulay for $\Omega = \mathbb{Z}_2[X,Y]/(X^2,Y^2)$.
\end{itemize}
\end{remark}

\subsection{$K(\Omega)$ Cannot Be a Compact Surface}

We have shown that for any finite commutative ring $\Omega$, the complex $K(\Omega)$ is not a simplicial decomposition of the Klein bottle or the real projective plane. Using Theorem \ref{main}, we can study if $K(\Omega)$ is a simplicial decomposition of a compact surface.

\begin{theorem}
For any commutative ring $\Omega$, the complex $K(\Omega)$ is not a simplicial decomposition of a compact real surface.
\end{theorem}
\begin{proof}
By Theorem \ref{main}, the $K(\Omega)$ can never be a nonorientable surface. If $K(\Omega)$ is the compact orientable surface $\Sigma_g$ with genus $g$, then we have $\tilde{H}_2(K(\Omega))=\mathbb{Z}$, $\tilde{H}_1(K(\Omega))=\mathbb{Z}^{2g}$ and $\dim\ K(\Omega)=2$. We consider the value of $k$ in the decomposition $\Omega\cong R_1 \times R_2 \times\cdots\times R_k$.\par

If $k\geq 4$, then $\dim\ K(\Omega)\geq k-1\geq 3$. If $k=1$ or $k=2$, we have $\tilde{H}_2(K(\Omega))=0$. The only possibility is $k=3$. When some $R_k$ is $\mathbb{Z}_2$, the homology group $\tilde{H}_2(K(\Omega))$ is 0. If some $R_k$ is not a field, then $b_2(K(\Omega))>1$. The remaining case is $\Omega\cong F_1\times F_2\times F_3$. Since $\tilde{H}_2(K(\Omega))=\mathbb{Z}^{(u_1-1)(u_2-1)(u_3-1)}$, the condition $\tilde{H}_2(K(\Omega))=\mathbb{Z}$ implies that $u_1=u_2=u_3=2$. This shows that if $K(\Omega)$ is a compact surface, it must be $\Sigma_6$, since $\tilde{H}_1(K(\mathbb{Z}_3\times\mathbb{Z}_3\times\mathbb{Z}_3))=\mathbb{Z}^{12}$. It is easy to check that $K(\mathbb{Z}_3\times\mathbb{Z}_3\times\mathbb{Z}_3)$ is an octahedral with 12 1-dimensional handles, which is not $\Sigma_6$.
\end{proof}

\section{Questions}

The main difficulty of extending our Theorem \ref{main} to finite semigroups is that, for a general semigroup $S$, there is no easy  decomposition theorem as in the ring case $\Omega\cong R_1 \times R_2 \times\cdots\times R_k$. In semigroup theory, the direct product is replaced by the structure of a semilattice of semigroups, but the structure of semilattice is also complicated. For finite noncommutative rings, the situation is not as bad as in the case of semigroups. We may study the following question:

\begin{question}
How to compute the homology group $K(\Omega)$ if $\Omega$ is a finite noncommutative ring?
\end{question}

In Theorem \ref{cm}, there is a case remains open.

\begin{question}
If $\Omega$ is a finite commutative local ring (not a field) with maximal ideal $\mathfrak{m}$ of nilpotency index $v$ and $|\mathfrak{m}^{v-1}| = 2$, find the necessary and sufficient condition that $K(\Omega)$ is Cohen-Macaulay.
\end{question}

\end{document}